\documentclass[11pt]{amsart}

\usepackage[utf8]{inputenc}
\usepackage[T1]{fontenc}

\usepackage{amsmath,amssymb,amsfonts,textcomp,amsthm,xifthen,graphicx,color,pgfplots}
\usepackage{ifthen}

\usepackage{enumerate}
\usepackage{enumitem}

\usepackage[utf8]{inputenc}

\usepackage[pdftex,
            pdfauthor={Thomas F\"uhrer and Juha Videman},
            pdftitle={FOSLS for singularly perturbed Darcy equations},
            ]{hyperref}

\usepackage{pgfplots}
\usepgfplotslibrary{groupplots}
\usepgfplotslibrary{colorbrewer}
\pgfplotsset{compat=newest}
\usepgfplotslibrary{external}
\usepackage{float} 
\usepackage{array}

\usepackage{multirow}

\usepackage{etoolbox}
\newcounter{rowcntr}[table]
\renewcommand{\therowcntr}{\alph{rowcntr})}

\newcolumntype{N}{>{\refstepcounter{rowcntr}\therowcntr}c}

\AtBeginEnvironment{tabular}{\setcounter{rowcntr}{0}}

\newtheorem{theorem}{Theorem}
\newtheorem{example}{Example}

\newtheorem{corollary}[theorem]{Corollary}
\newtheorem{proposition}[theorem]{Proposition}
\newtheorem{remark}[theorem]{Remark}

\newcommand{\err}{\operatorname{err}}

\newcommand{\logLogSlopeTriangle}[6]
{

    \pgfplotsextra
    {
        \pgfkeysgetvalue{/pgfplots/xmin}{\xmin}
        \pgfkeysgetvalue{/pgfplots/xmax}{\xmax}
        \pgfkeysgetvalue{/pgfplots/ymin}{\ymin}
        \pgfkeysgetvalue{/pgfplots/ymax}{\ymax}

        \pgfmathsetmacro{\xArel}{#1}
        \pgfmathsetmacro{\yArel}{#3}
        \pgfmathsetmacro{\xBrel}{#1-#2}
        \pgfmathsetmacro{\yBrel}{\yArel}
        \pgfmathsetmacro{\xCrel}{\xArel}

        \pgfmathsetmacro{\lnxB}{\xmin*(1-(#1-#2))+\xmax*(#1-#2)} 
        \pgfmathsetmacro{\lnxA}{\xmin*(1-#1)+\xmax*#1} 
        \pgfmathsetmacro{\lnyA}{\ymin*(1-#3)+\ymax*#3} 
        \pgfmathsetmacro{\lnyC}{\lnyA+#4*(\lnxA-\lnxB)}
        \pgfmathsetmacro{\yCrel}{\lnyC-\ymin)/(\ymax-\ymin)} 

        \coordinate (A) at (rel axis cs:\xArel,\yArel);
        \coordinate (B) at (rel axis cs:\xBrel,\yCrel);
        \coordinate (C) at (rel axis cs:\xCrel,\yCrel);

        \draw[#5]   (A)--
                    (B)-- 
                    (C)-- node[pos=0.5,anchor=west] {#6}
                    cycle;
    }
}

\def\enorm#1{|\hspace*{-.5mm}|\hspace*{-.5mm}|#1|\hspace*{-.5mm}|\hspace*{-.5mm}|}

\newcommand{\ip}[2]{(#1\hspace*{.5mm},#2)}

\newcommand{\norm}[3][]{#1\|#2#1\|_{#3}}

\newcommand{\diam}{\mathrm{diam}}

\def\div{{\rm div\,}}
\DeclareMathOperator{\Div}{\mathbf{div}}
\def\Dev{{\mathbf{Dev}\,}}

\newcommand{\HDivset}[1]{\underbar\HH(\Div;#1)}
\newcommand{\HDivsetStar}[1]{\underbar\HH_*(\Div;#1)}

\newcommand{\di}{\mathrm{d}}

\newcommand{\set}[2]{\big\{#1\,:\,#2\big\}}

\newcommand{\RT}{\ensuremath{\mathcal{RT}}}

\newcommand{\R}{\ensuremath{\mathbb{R}}}
\newcommand{\N}{\ensuremath{\mathbb{N}}}
\newcommand{\HH}{\ensuremath{{\boldsymbol{H}}}}

\newcommand{\LL}{\ensuremath{\boldsymbol{L}}}
\newcommand{\LLL}{\underbar\LL}

\newcommand{\vv}{\ensuremath{\boldsymbol{v}}}

\newcommand{\TT}{\ensuremath{\mathcal{T}}}
\newcommand{\FF}{\ensuremath{\boldsymbol{F}}}

\newcommand{\cS}{\ensuremath{\mathcal{S}}}
\newcommand{\cSS}{\boldsymbol\cS}

\newcommand{\PP}{\ensuremath{\mathcal{P}}}
\newcommand{\OO}{\ensuremath{\mathcal{O}}}

\newcommand{\NN}{\ensuremath{\boldsymbol{N}}}
\newcommand{\zz}{\ensuremath{{\boldsymbol{z}}}}

\newcommand{\GG}{\ensuremath{\boldsymbol{G}}}
\newcommand{\ff}{\ensuremath{\boldsymbol{f}}}
\newcommand{\bg}{\ensuremath{\boldsymbol{g}}}

\newcommand\Grad{\boldsymbol\nabla}

\newcommand{\pphi}{{\boldsymbol\phi}}

\newcommand{\uu}{\boldsymbol{u}}

\newcounter{constantsnumber}
\def\setc#1{
  \ifthenelse{\equal{#1}{poinc}}{C_{\rm edge}}{ 
   \refstepcounter{constantsnumber}
   \label{const#1}C_{\theconstantsnumber}}}

\def\c#1{
  \ifthenelse{\equal{#1}{poinc}}{C_{\rm edge}}{ 
    C_{\ref{const#1}}}}

\newcommand{\MM}{\boldsymbol{M}}
\newcommand{\II}{\boldsymbol{I}}

\newcommand{\tr}{\operatorname{tr}}

\newcommand{\xx}{{\boldsymbol{x}}}

\begin{document}

\title[FOSLS for singularly perturbed Darcy equations]{First-order system least-squares finite element method for singularly perturbed Darcy equations}
\date{\today}

\author{Thomas F\"{u}hrer}
\address{Facultad de Matem\'{a}ticas, Pontificia Universidad Cat\'{o}lica de Chile, Santiago, Chile}
\email{tofuhrer@mat.uc.cl}

\author{Juha Videman}
\address{CAMGSD/Departamento de Matemática, Instituto Superior Técnico, Universidade de Lisboa,
1049-001 Lisbon, Portugal}
\email{jvideman@math.tecnico.ulisboa.pt}

\thanks{{\bf Acknowledgment.} 
This work was supported by FONDECYT project 1210391.}

\keywords{Least-squares finite element method, Brinkman equations, Darcy equations, singularly perturbed problem, first-order formulation}
\subjclass[2010]{65N30, 
                 65N12} 
\begin{abstract}
We define and analyse a least-squares finite element method for a first-order reformulation of a scaled Brinkman model of
fluid flow through porous media.
We introduce a pseudostress variable that allows to eliminate the pressure variable from the system. 
It can be recovered by a simple post-processing. It is shown that the least-squares functional is uniformly equivalent, i.e., independent of the singular perturbation parameter, 
to a parameter dependent norm. 
This norm equivalence implies that the least-squares functional evaluated in the discrete solution provides an efficient and reliable a posteriori error estimator.
Numerical experiments are presented.
\end{abstract}
\maketitle

\section{Introduction}

Let $\Omega \subset \R^d$ ($d=2,3$) be a bounded polytopal domain with boundary $\Gamma:=\partial\Omega$.
We consider the following set of equations: 
Given $\ff\in \LL^2(\Omega)$ find the velocity $\uu\in \HH^1(\Omega)$ and pressure $p\in L^2(\Omega)$ satisfying
\begin{subequations}\label{eq:model}
\begin{align}
  -t^2\boldsymbol\Delta \uu + \uu + \nabla p &=\ff \quad\text{in } \Omega, \\
  \div\uu &=0 \quad\text{in } \Omega, \\
  \int_\Omega p \,d\xx &=0, \\
  \uu|_\Gamma &= 0,
\end{align}
\end{subequations}
Throughout this article we assume that $t\in(0,1]$.
We note that for positive $t$ the model problem is of Stokes type and that for $t=0$ equations~\eqref{eq:model} describe the Darcy flow.
Problem~\eqref{eq:model} can thus be interpreted as perturbed Darcy equations where $t$ is the effective viscosity.
The set of equations~\eqref{eq:model} is obtained as a rescaled version of the Brinkman model of fluid flow through porous media.
For details on the derivation we refer to, e.g.,~\cite[Sec.~2]{HannukainenJuntunenStenberg13}.
As $t$ gets smaller, boundary layers develop so that robust numerical schemes are necessary. 

Several numerical methods have been developed for the Brinkman equations including mixed finite element methods, stabilized methods, HDG methods, virtual element methods, novel finite element spaces, see, e.g.,~\cite{JuntunenStenberg10,KonnoStenberg11,KonnoStenberg12,GuzmanNeilan12,BC09,BH07,VassilevskiVilla14,GaticaGaticaMarquez14,GaticaGaticaSequeira15,AnayaMoraOyarzuaRuizBaier16,CaceresGaticaSequeira17,Vacca18,GaticaSequeira18,BottiDiPietroDroniou18,MTW02}.

The main objective of this article is to develop and investigate a least-squares finite element method (LSFEM) for solving~\eqref{eq:model}.
Though LSFEMs for Stokes and Stokes-type problems are known, see, e.g.,~\cite{CaiLeeWang04} for the Stokes problem or~\cite{CaiChen16} for the Oseen equations, the situation is different for singularly perturbed problems like~\eqref{eq:model}. 
To the best of our knowledge the work at hand is the first work that provides an analysis of a first-order system least-squares finite element method for the Brinkman problem where all equivalence estimates are independent of the singular perturbation parameter $t\in(0,1]$.
In~\cite{DS07,MS11} LSFEMs for the (coupled) Stokes--Darcy flow are considered.

Some basic features of LSFEMs that motivate their study for problem~\eqref{eq:model} are: First, and, contrary to many mixed FEMs, any choice of discretization spaces is $\inf$--$\sup$ stable.
Second, the discretized set of equations lead to symmetric, positive definite matrices, thus, standard iterative solvers like the (preconditioned) conjugate gradient method can be applied. 
Third, LSFEMs usually come with a ``built-in'' reliable and efficient error estimator that can be used to steer an adaptive algorithm. 
For an overview on the general theory of LSFEMs we refer to the textbook~\cite{BochevGunzburgerLSQ}, and particular~\cite[Chapter~7]{BochevGunzburgerLSQ} for LSFEMs for the Stokes problem. 

To define our LSFEM we consider first a reformulation by introducing a pseudostress variable, see, e.g.,~\cite{CaiLeeWang04} for a pseudostress reformulation of the Stokes problem. 
This yields a first order formulation of problem~\eqref{eq:model} which also allows to eliminate the pressure variable from the system. 
We then define our least-squares functional based on the reduced formulation. 
The main contributions of this work are the analysis of a general form of the first-order reformulation and the norm equivalence between the least-squares functional and a canonical energy norm. 
We stress that a careful analysis is required to prove that the estimates are independent of the singular perturbation parameter $t$.
Concerning discretization spaces we are using Lagrange finite elements for the approximation of the velocity and Raviart--Thomas finite elements for the approximation of the pseudostress. This generic choice might lead to an undesired locking effect as demonstrated in Example~\ref{ex:convrates} below but can be removed by augmenting the discretization space for the pseudostress with deviatoric-free elements, see Example~\ref{ex:convrates2}.

\subsection{Overview}
The remainder of this work is given as follows:
In Section~\ref{sec:fo} we introduce notation and a first-order reformulation of our model problem~\eqref{eq:model}. 
Moreover, a reduced first-order formulation is given and analyzed in Theorem~\ref{thm:fo:reduced}. 
In Section~\ref{sec:lsq} we define the least-squares method and prove norm equivalence of the least-squares functional to a canonical parameter dependent norm.
The main results are presented in Theorem~\ref{thm:lsq:J} and Corollary~\ref{cor:lsq:J}. 
Approximation spaces are discussed in Section~\ref{sec:discretization} and Section~\ref{sec:num} concludes this article with several numerical experiments.

\section{First-order formulation}\label{sec:fo}

\subsection{Notation}
Vector valued quantities are denoted with boldfaced letters, matrix valued functions ($d\times d$ tensors) with
boldfaced capital letters.
Throughout $\II$ denotes the $d\times d$ identity matrix, $\tr : \R^{d\times d} \to \R$ is the trace operator of a
matrix. Furthermore, $\Dev(\cdot)$ denotes the deviatoric part of a matrix, i.e.,
\begin{align*}
  \Dev\MM = \MM - \frac1{d} \tr(\MM)\II.
\end{align*}

We use the common notation for Lebesgue and Sobolev spaces, $L^2(\Omega)$, $H^1(\Omega)$.
The space $H_0^1(\Omega)$ is the space of $H^1(\Omega)$ functions with zero trace on $\Gamma$.
The canonical $L^2(\Omega)$ inner product is denoted by $\ip\cdot\cdot$ and the induced norm by $\norm\cdot{}$. 
Additionally, we define
\begin{align*}
  L_*^2(\Omega) := \set{v\in L^2(\Omega)}{\ip{v}1 = 0}.
\end{align*}
We use boldfaced
symbols for product spaces like
\begin{align*}
  \LL^2(\Omega) := L^2(\Omega)^d, \quad  \HH^1(\Omega) := H^1(\Omega)^d, 
  \quad \HH_0^1(\Omega) := H_0^1(\Omega)^d.
\end{align*}

Spaces for matrix valued functions are denoted by boldfaced capital letters with an underscore, e.g.,
\begin{align*}
  \LLL^2(\Omega) := L^2(\Omega)^{d\times d}.
\end{align*}
Furthermore, 
\begin{align*}
  \HDivset\Omega :=\set{\MM\in \LLL^2(\Omega)}{\Div\MM\in \LL^2(\Omega)},
\end{align*}
where $\Div$ denotes the row-wise divergence operator.
We also use the divergence operator $\div$, the gradient $\nabla$ and the vector gradient $\Grad$ which is defined by
applying $\nabla$ to each element, e.g., for $d=2$ we have that
\begin{align*}
  \Grad\uu = \Grad\begin{pmatrix}\uu_1 \\ \uu_2\end{pmatrix}
  = \begin{pmatrix}
    \nabla\uu_1^\top \\
    \nabla\uu_2^\top
  \end{pmatrix}
  = \begin{pmatrix}
    \frac{\partial \uu_1}{\partial \xx_1} & \frac{\partial \uu_1}{\partial \xx_2} \\
    \frac{\partial \uu_2}{\partial \xx_1} & \frac{\partial \uu_2}{\partial \xx_2}
  \end{pmatrix}.
\end{align*}
The vector Laplacian $\boldsymbol\Delta$ is defined as applying $\Delta$ to each component, i.e.,
$(\boldsymbol\Delta\uu)_j = \Delta\uu_j$ for $j=1,\dots,d$.

For the inner products and norms in $\LL^2(\Omega)$ and $\LLL^2(\Omega)$ we use the same notations as in the scalar
case, i.e., for $\uu,\vv\in\LL^2(\Omega)$, $\MM,\NN\in \LLL^2(\Omega)$ we have that
\begin{align*}
  \ip{\uu}{\vv} &:= \int_\Omega \uu\cdot\vv\,d\xx = \sum_{j=1}^d \ip{\uu_j}{\vv_j}, \quad \norm{\uu}{}^2 := \ip{\uu}{\uu}, \\
  \ip{\MM}{\NN} &:= \int_\Omega \MM : \NN \,d\xx = \sum_{j,k=1}^d \ip{\MM_{jk}}{\NN_{jk}}, \quad
  \norm{\MM}{}^2 := \ip{\MM}{\MM}.
\end{align*}

Basic manipulations show
\begin{align*}
  \norm{\MM}{}^2 = \norm{\Dev\MM}{}^2 + \frac1{d} \norm{\tr\MM}{}^2,
\end{align*}
and thus $\norm{\tr\MM}{} \lesssim \norm{\MM}{}$ and $\norm{\Dev\MM}{}\leq \norm{\MM}{}$.

For $T,S\geq 0$ we write $T\lesssim S$ if there exists $C>0$ such that $T\leq C\, S$. If $T\lesssim S$ and $S\lesssim T$ then we write $T\eqsim S$.

\subsection{Pseudostress formulation}
To obtain a first-order reformulation we consider the pseudostress
\begin{align*}
  \MM := t \Grad\uu - \frac{p}t\II.
\end{align*}
Here, we choose to work with the pseudostress instead of the symmetric stress (replacing $\Grad\uu$ with its symmetric
part) because it allows for a simple algebraic postprocessing to obtain other physical meaningful quantities like the
vorticity. This was elaborated in, e.g.,~\cite{CaiLeeWang04}.

The first-order reformulation of~\eqref{eq:model} is then given by
\begin{subequations}\label{eq:fo}
\begin{align}
  -t\Div\MM + \uu &= \ff, \label{eq:fo:a}\\
  \MM-t\Grad\uu+\frac{p}t\II &= 0, \label{eq:fo:b}\\
  \div\uu &= 0, \label{eq:fo:c}\\
  \int_\Omega p \,d\xx &= 0, \\
  \uu|_\Gamma &=0.
\end{align}
\end{subequations}

\subsection{Analysis of the general first-order system}
For the analysis of the first-order system we work with the parameter-dependent norm
\begin{align*}
  \norm{\uu}t^2 := \norm{\uu}{}^2 + t^2\norm{\Grad\uu}{}^2 + \norm{\div\uu}{}^2
\end{align*}
for the velocity variable, the canonical $L^2(\Omega)$ norm for the pressure variable, and the parameter-dependent norm
\begin{align*}
  \norm{\MM}t^2 := \norm{\Dev\MM}{}^2 + t^2\norm{\Div\MM}{}^2 + t^2\norm{\tr\MM}{}^2
\end{align*}
for the pseudostress variable.

We equip the space $X := \HH_0^1(\Omega)\times \HDivset\Omega$ with the parameter-dependent norm
\begin{align*}
  \enorm{(\uu,\MM)}_t^2 := \norm{\uu}t^2 + \norm{\MM}t^2.
\end{align*}
Note that the product space 
\begin{align*}
  X_* := \HH_0^1(\Omega)\times \HDivsetStar\Omega := \HH_0^1(\Omega) \times (\HDivset\Omega\cap \LLL_*^2(\Omega)).
\end{align*}
where $\LLL_*^2(\Omega) := \set{\NN\in\LLL^2(\Omega)}{\int_\Omega \tr(\NN)\,\di\xx=0}$, 
is a closed subspace with respect to the norm $\enorm\cdot_t$.
For our analysis we also consider a general system including the pressure variable where we need the space
\begin{align*}
  Y_*&:= \HH_0^1(\Omega)\times \HDivset\Omega \times L_*^2(\Omega).
\end{align*}

The next result follows from the well-known continuous $\inf$--$\sup$ condition for the Stokes problem:
\begin{proposition}\label{prop:bb}
  There exists a constant $C>0$ which depends only on $\Omega$ such that for all $t\in (0,1]$
  \begin{align*}
    C \norm{p}{} \leq \sup_{\uu\in \HH_0^1(\Omega)\setminus\{0\}} \frac{\ip{p}{\div\uu}}{\norm{\uu}t} 
    \quad\text{for all } p\in L_*^2(\Omega).
  \end{align*}
\end{proposition}
\begin{proof}
  It is well-known, see e.g.~\cite[Ch.~8, Sec.~2]{BoffiBrezziFortin}, that
  \begin{align*}
    \sup_{\uu\in \HH_0^1(\Omega)\setminus\{0\}} \frac{\ip{p}{\div\uu}}{\norm{\uu}{}+\norm{\Grad\uu}{}} \gtrsim
    \norm{p}{}.
  \end{align*}
  Together with the norm equivalence $\norm{\uu}{}+\norm{\Grad\uu}{}\simeq \norm{\uu}1$ and the estimate
  $\norm{\uu}t\leq \norm{\uu}1$ uniformly in $t\in(0,1]$, this finishes the proof.
\end{proof}

In the next result we analyze a general form of~\eqref{eq:fo}.
\begin{theorem}\label{thm:fo}
  Let $(\ff,\FF,f)\in \LL^2(\Omega)\times \LLL^2(\Omega) \times L_*^2(\Omega)$ be given.
  The problem
  \begin{subequations}\label{eq:fo:general}
  \begin{align}
    -t\Div\MM + \uu &= \ff, \label{eq:fo:general:a} \\
    \MM-t\Grad\uu+\frac{p}t\II &= \FF, \label{eq:fo:general:b} \\
    \div\uu &= f. \label{eq:fo:general:c}
  \end{align}
  \end{subequations}
  admits a unique solution $(\uu,\MM,p)\in Y_*$.
  Moreover, there exists a constant $C>0$ independent of $t\in(0,1]$ such that
  \begin{align*}
    \enorm{(u,\MM)}_t+\norm{p}{} \leq C (\norm{\ff}{} + \norm{\Dev\FF}{} + t\norm{\tr\FF}{} + \norm{f}{}).
  \end{align*}
\end{theorem}
\begin{proof}
  We split the proof into several steps.
  
  \noindent
  \textbf{Step 1.}
  We consider problem~\eqref{eq:model} with general data
  \begin{align*}
    -t^2\boldsymbol\Delta \uu + \uu +\nabla p &= \ff + t\Div\FF, \\
    \div\uu &= f
  \end{align*}
  which admits a unique weak solution $(\uu,p)\in \HH_0^1(\Omega)\times L_*^2(\Omega)$, i.e., $(\uu,p)$ is the solution of
  \begin{align}\label{eq:fo:proof:a}
    \begin{split}
    &t^2\ip{\Grad\uu}{\Grad\vv} + \ip{\uu}{\vv} -\ip{p}{\div\vv} -\ip{q}{\div\uu} 
    \\&\qquad\qquad= \ip{\ff}{\vv} + t\ip{\Div\FF}{\vv}
    -\ip{f}q
  \end{split}
  \end{align}
  for all $(\vv,p)\in \HH_0^1(\Omega)\times L_*^2(\Omega)$.
  For our analysis it is more convenient to use an equivalent weak formulation that is obtained by testing the second
  equation $\div\uu=f$ with $\div\vv$ which leads to
  \begin{align}\label{eq:fo:proof:b}
    \begin{split}
    &a_t(\uu,\vv) -\ip{p}{\div\vv} -\ip{q}{\div\uu} 
    \\&\qquad\qquad= \ip{\ff}{\vv} + t\ip{\Div\FF}{\vv}
    -\ip{f}q + \ip{f}{\div\vv}
  \end{split}
  \end{align}
  for all $(\vv,p)\in \HH_0^1(\Omega)\times L_*^2(\Omega)$. Here, the bilinear form $a_t(\cdot,\cdot)$ is defined by
  \begin{align*}
    a_t(\uu,\vv) := t^2\ip{\Grad\uu}{\Grad\vv} + \ip{\uu}\vv + \ip{\div\uu}{\div\vv}
    \quad\text{for all }\uu,\vv\in \HH_0^1(\Omega).
  \end{align*}
  To see that~\eqref{eq:fo:proof:a} and~\eqref{eq:fo:proof:b} are equivalent we test either equation with $(0,q)$ to get
  that $\div\uu = f$. This shows that any solution $(\uu,p)$ of~\eqref{eq:fo:proof:a} solves~\eqref{eq:fo:proof:b} and
  vice versa.

  \noindent
  \textbf{Step 2.}
  We analyse the dependence of the unique solution $(\uu,p)$ of~\eqref{eq:fo:proof:b} on the data:
  First, note that $a_t(\cdot,\cdot)$ is the inner product that induces the norm $\norm{\cdot}t$.
  Second,
  \begin{align*}
    |\ip{q}{\div\vv}| \leq \norm{q}{}\norm{\vv}t.
  \end{align*}
  Together with Proposition~\ref{prop:bb} the well-known Babu\v{s}ka--Brezzi theory shows that the unique solution of
  problem~\eqref{eq:fo:proof:b} satisfies
  \begin{align*}
    \norm{\uu}t + \norm{p}{} &\lesssim \sup_{0\neq(\vv,q)\in \HH_0^1(\Omega)\times L_*^2(\Omega)} 
    \frac{\ip{\ff}{\vv} + t\ip{\Div\FF}{\vv}-\ip{f}q + \ip{f}{\div\vv}}{\norm{\vv}t+\norm{q}{}} \\
    &\lesssim \norm{\ff}{} + \norm{f}{} + \sup_{0\neq\vv\in \HH_0^1(\Omega)} 
    \frac{t\ip{\Div\FF}{\vv}}{\norm{\vv}t}.
  \end{align*}
  To estimate the last term we use integration by parts and then split $\FF$ into its deviatoric and trace parts
  respectively, i.e., $\FF = \Dev\FF + d^{-1}\tr(\FF)\II$. This and the fact that $\II:\Grad\vv = \div\vv$ leads to
  \begin{align*}
    t|\ip{\Div\FF}{\vv}| &= t|\ip{\FF}{\Grad\vv}| \leq t|\ip{\Dev\FF}{\Grad\vv}| 
    + t|\ip{d^{-1}\tr(\FF)\II}{\Grad\vv}| \\
    &= t|\ip{\Dev\FF}{\Grad\vv}| 
    + t|\ip{d^{-1}\tr{\FF}}{\div\vv}| 
    \\
    &\leq \norm{\Dev\FF}{}t\norm{\Grad\vv}{} + t\norm{\tr\FF}{}\norm{\div\vv}{}
    \\
    &\leq \left(\norm{\Dev\FF}{}^2+t^2\norm{\tr\FF}{}^2\right)^{1/2}\norm{\vv}t.
  \end{align*}
  Putting all estimates together shows that
  \begin{align*}
    \norm{\uu}t + \norm{p}{} \lesssim \norm{\ff}{} + \norm{f}{} + \norm{\Dev\FF}{}+t\norm{\tr\FF}{}.
  \end{align*}

  \noindent
  \textbf{Step 3.}
  Let $(\uu,p)\in \HH_0^1(\Omega)\times L_*^2(\Omega)$ be the solution of~\eqref{eq:fo:proof:b}.
  Define $\MM$ by relation~\eqref{eq:fo:general:b}, i.e.,
  \begin{align*}
    \MM:= t\Grad\uu - \frac{p}t\II +\FF.
  \end{align*}
  Then, $\MM\in \LLL^2(\Omega)$ and it remains to show that $\Div\MM\in \LL^2(\Omega)$ and that $\MM$
  satisfies~\eqref{eq:fo:general:a}:
  With $\pphi\in \HH_0^1(\Omega)$ we get using~\eqref{eq:fo:proof:a}, $\div\uu=f$, and integration by parts that
  \begin{align*}
    -t\ip{\Div\MM}{\pphi} &= t\ip{\MM}{\Grad\pphi} = t^2\ip{\Grad\uu}{\Grad\pphi} - \ip{p}{\div\pphi} +
    t\ip{\FF}{\Grad\pphi} 
    \\&= \ip{-t^2\boldsymbol{\Delta}\uu+\nabla p-t\Div\FF}{\pphi} = \ip{\ff-\uu}{\pphi}.
  \end{align*}
  This proves that $\div\MM = t^{-1}(\uu-\ff)\in \LL^2(\Omega)$ as well as $\MM$ satisfies~\eqref{eq:fo:general:a}. Thus, $(\uu,\MM,p)\in Y_*$ is a
  solution of the first-order system~\eqref{eq:fo:general}.

  \noindent
  \textbf{Step 4.}
  Let $(\uu,\MM,p)\in Y_*$ be given as in Step~3.
  Then,
  \begin{align*}
    \norm{\Dev\MM}{} \leq t\norm{\Dev\Grad\uu}{} + \norm{\Dev\FF}{} \lesssim \norm{\uu}t + \norm{\Dev\FF}{},
  \end{align*}
  and
  \begin{align*}
    t\norm{\tr\MM}{} \lesssim t^2\norm{\div\uu}{} + \norm{p}{} + t\norm{\tr\FF}.
  \end{align*}
  Together with the final estimate from Step~2 this shows that
  \begin{align*}
    \enorm{(\uu,\MM)}_t + \norm{p}{}\lesssim \norm{\ff}{} + \norm{\Dev\FF}{} 
    +t\norm{\tr\FF}{} + \norm{f}{}.
  \end{align*}

  \noindent
  \textbf{Step 5.}
  It remains to show that solutions of~\eqref{eq:fo:general} are unique:
  Suppose that $(\uu,\MM,p)\in Y_*$ solves~\eqref{eq:fo:general} with $(\ff,\FF,f) = 0$.
  Then, using~\eqref{eq:fo:general:a} and~\eqref{eq:fo:general:b} we obtain that
  \begin{align*}
    -t\ip{\Div\MM}{\vv} + \ip{\uu}\vv &= t\ip{\MM}{\Grad\vv}+\ip{\uu}\vv 
    \\
    &= t^2\ip{\Grad\uu}{\Grad\vv} +\ip{\uu}\vv
    -\ip{p}{\div\uu} = 0
  \end{align*}
  for all $\vv\in\HH_0^1(\Omega)$.
  Together with~\eqref{eq:fo:general:c} this shows that $(\uu,p)$ solves~\eqref{eq:fo:proof:a} with vanishing right-hand
  side. Consequently, $(\uu,p) = 0$. Finally,~\eqref{eq:fo:general:b} implies that $\MM=0$ which finishes the proof.
\end{proof}

The last result indicates to consider the least-squares functional
\begin{align*}
  &\norm{-t\Div\MM+\uu-\ff}{}^2 \!+\! \norm{\Dev(\MM-t\Grad\uu)}{}^2 
  \\
  &\qquad
  \!+\! \norm{t\tr\MM-t^2\div\uu+dp}{}^2 \!+\! \norm{\div\uu}{}^2
\end{align*}
However, we can also eliminate the pressure from the system~\eqref{eq:fo:general} and obtain a reduced first-order formulation.
Consider the solution $(\uu,\MM,p)$ from~\eqref{eq:fo:general} with $(\ff,\FF,f) = (\ff,0,0)$. Taking the trace of~\eqref{eq:fo:general:b} and using that $\div\uu = 0$ we infer that
\begin{align*}
  \tr\MM -t\div\uu +d t^{-1}p = \tr\MM+ dt^{-1}p = 0.
\end{align*}
Moreover, condition $\int_\Omega p\,\di\xx = 0$ corresponds to $\int_\Omega \tr\MM \,\di\xx =0$. Thus, the resulting problem is: Find $(\uu,\MM) \in X_*$ such that
\begin{subequations}\label{eq:fo:reduced}
\begin{align}
  -t\Div\MM + \uu &= \ff, \\
  \Dev\MM -t\Grad\uu &= 0.
\end{align}
\end{subequations}
We analyze the general form of~\eqref{eq:fo:reduced} in the next result.
\begin{theorem}\label{thm:fo:reduced}
  Let $(\bg,\GG)\in \LL^2(\Omega)\times \LLL^2_*(\Omega)$ be given. The system 
  \begin{subequations}\label{eq:fo:general:reduced}
  \begin{align}
    -t\Div\MM + \uu &= \bg,\label{eq:fo:general:reduced:a} \\
    \Dev\MM-t\Grad \uu &= \GG \label{eq:fo:general:reduced:b}
  \end{align}
  \end{subequations}
  admits a unique solution $(\uu,\MM)\in\HH_0^1(\Omega)\times \HDivsetStar\Omega$. 

  Furthermore, there exists a constant $C>0$ independent of $t\in(0,1]$ such that 
  \begin{align}
    \enorm{(\uu,\MM)}_t \leq C(\norm{\bg}{} + \norm{\Dev\GG}{} + t^{-1}\norm{\tr\GG}{}).
  \end{align}
\end{theorem}
\begin{proof}
  We show existence of solutions of~\eqref{eq:fo:general:reduced}: 
  Let $(\bg,\GG)\in \LL^2(\Omega)\times \LLL^2_*(\Omega)$ be given. Consider the problem
  \begin{subequations}\label{eq:fo:general:reduced:proof}
  \begin{align}
    -t\Div\MM + \uu &= \bg, \label{eq:fo:general:reduced:proof:a} \\
    \MM - t\Grad\uu + \frac{p}t\II &= \Dev\GG, \label{eq:fo:general:reduced:proof:b} \\
    \div\uu &= -\frac1{t} \tr\GG. \label{eq:fo:general:reduced:proof:c}
  \end{align}
  \end{subequations}
  Using $(\ff,\FF,f) := (\bg,\Dev\GG,-t^{-1}\tr\GG)\in \LL^2(\Omega)\times \LLL^2(\Omega)\times L^2_*(\Omega)$ we immediately get from Theorem~\ref{thm:fo} that the latter system admits a unique solution satisfying the estimate
  \begin{align*}
    \enorm{(\uu,\MM)}_t \lesssim \norm{\bg}{} + \norm{\Dev\GG}{} + t^{-1}\norm{\tr\GG}{}.
  \end{align*}
  Furthermore, by taking the trace of~\eqref{eq:fo:general:reduced:proof:b} and replacing $\div\uu$ by the third equation we represent $p$ as
  \begin{align*}
    \frac{p}t = -\frac1d \tr\MM + \frac{t}d\div\uu = -\frac1d\tr\MM - \frac{1}d \tr\GG.
  \end{align*}
  Integrating over $\Omega$ and using that $\tr\GG\in L_*^2(\Omega)$ and $p\in L_*^2(\Omega)$ shows that $\tr\MM\in L_*^2(\Omega)$, i.e., $\MM\in \HDivsetStar\Omega$. 
  Using the latter identity another time
  and $\Dev\MM = \MM-d^{-1}\tr(\MM)\II$ in~\eqref{eq:fo:general:reduced:proof:b} we get that
  \begin{align*}
    \Dev\MM - t\nabla \uu -\frac{1}d(\tr\GG)\II = \Dev\GG,
  \end{align*}
  thus, 
  \begin{align*}
    \Dev\MM-t\nabla \uu = \GG.
  \end{align*}
  This, together with~\eqref{eq:fo:general:reduced:proof:b} shows that the pair $(\uu,\MM)\in \HH_0^1(\Omega)\times\HDivsetStar\Omega$ solves~\eqref{eq:fo:general:reduced}. 

  To see uniqueness of solutions suppose that $(\uu,\MM)\in \HH_0^1(\Omega)\times \HDivsetStar\Omega$ solves~\eqref{eq:fo:general:reduced} with $(\bg,\GG) = (0,0)$. It suffices to prove that this implies $(\uu,\MM)=(0,0)$.
  Taking the trace of~\eqref{eq:fo:general:reduced:b} we deduce that $\div\uu = 0$. Then, the norms of the residuals in~\eqref{eq:fo:general:reduced} and integration by parts imply that
  \begin{align*}
    0 &= \norm{-t\Div\MM + \uu}{}^2 + \norm{\Dev\MM-t\Grad\uu}{}^2 
    \\
    &= t^2 \norm{\Div\MM}{}^2 -2t\ip{\Div\MM}{\uu} + \norm{\uu}{}^2 
    \\&\qquad + \norm{\Dev\MM}{}^2 -2t\ip{\MM}{\Grad\uu} - 2td^{-1}\ip{\tr\MM}{\div\uu} + t^2\norm{\Grad\uu}{}^2
    \\
    &= t^2 \norm{\Div\MM}{}^2 + \norm{\Dev\MM}{}^2 + t^2\norm{\Grad\uu}{}^2 + \norm{\uu}{}^2.
  \end{align*}
  We conclude that $\uu = 0$, $\Dev\MM=0$ and $\Div\MM=0$. From $\Dev\MM=0$ it follows that $\MM=d^{-1}\tr(\MM)\II$. Then, $\Div\MM=0$ yields $\nabla\tr(\MM) = 0$.
  Since $\int_\Omega\tr\MM\,\di \xx = 0$ we have that $\tr\MM=0$, thus, $\MM=0$. 
  This finishes the proof.
\end{proof}

\section{Least-squares finite element method}\label{sec:lsq}
Based on the first-order reformulation~\eqref{eq:fo:reduced} and Theorem~\ref{thm:fo:reduced} we consider the functional
\begin{align*}
  J_*(\uu,\MM;\ff) := \norm{-t\Div\MM+\uu-\ff}{}^2 + \norm{\Dev\MM-t\Grad\uu}{}^2 + \norm{\div\uu}{}^2
\end{align*}
for all $(\uu,\MM)\in X_*$. Note that $\norm{\Dev\MM-t\Grad\uu}{}^2 = \norm{\Dev(\MM-t\Grad\uu)}{}^2 + d^{-1}t^2\norm{\div\uu}{}^2$.
The following theorem is one of our main results:
\begin{theorem}\label{thm:lsq}
  There exists a constant $C>0$ which depends on $\Omega$ but is independent of $t\in(0,1]$ such that
  \begin{align}
    C^{-1} \enorm{(\uu,\MM)}_t^2 \leq J_*(\uu,\MM;0) \leq
    C \enorm{(\uu,\MM)}_t^2
    \quad\text{for all } (\uu,\MM)\in X_*.
  \end{align}
\end{theorem}
\begin{proof}
  The upper bound follows by the triangle inequality, i.e.,
  \begin{align*}
    J_*(\uu,\MM;0) &\lesssim t^2\norm{\Div\MM}{}^2 + \norm{\uu}{}^2 + 
    \norm{\Dev\MM}{}^2 + t^2\norm{\Grad \uu}{}^2 + \norm{\div\uu}{}^2
    \\ &\leq \enorm{(\uu,\MM)}_t^2.
  \end{align*}

  For the proof of the lower bound we apply Theorem~\ref{thm:fo:reduced}: Let $(\uu,\MM)\in X_*$ be given and define
  \begin{align*}
    \bg := -t\Div\MM +\uu, \quad \GG := \Dev\MM-t\Grad\uu,
  \end{align*}
  Note that $(\bg,\GG)\in \LL^2(\Omega)\times \LL_*^2(\Omega)$.
  By Theorem~\ref{thm:fo:reduced} the solution of the first-order system~\eqref{eq:fo:general:reduced} is unique and is therefore given
  by $(\uu,\MM)$. In particular, Theorem~\ref{thm:fo:reduced} shows that
  \begin{align*}
    \enorm{(\uu,\MM)}_t^2 &\lesssim \norm{\bg}{}^2 + \norm{\Dev\GG}{}^2 + t^{-2}\norm{\tr{\GG}}{}^2
    \\
    &= \norm{-t\Div\MM+\uu}{}^2 + \norm{\Dev(\MM-t\Grad\uu)}{}^2 + \norm{\div\uu}{}^2 
    \\
    &\leq \norm{-t\Div\MM+\uu}{}^2 + \norm{\Dev(\MM-t\Grad\uu)}{}^2 
    \\ &\qquad+ d^{-1}t^2\norm{\div\uu}{}^2 + \norm{\div\uu}{}^2 
    \\ 
    &= \norm{-t\Div\MM+\uu}{}^2 + \norm{\Dev\MM-t\Grad\uu}{}^2 + \norm{\div\uu}{}^2 
    \\ &= J_*(\uu,\MM;0).
  \end{align*}
  This concludes the proof.
\end{proof}

The following result is an immediate consequence of Theorem~\ref{thm:lsq} by using the well-established theory of least-squares principles, see, e.g.,~\cite[Section~2.2.1]{BochevGunzburgerLSQ} and in particular~\cite[Theorem~2.5 and Section~12.14]{BochevGunzburgerLSQ}.
\begin{corollary}\label{cor:lsq}
  Let $\ff\in\LL^2(\Omega)$ be given. The minimization problem
  \begin{align*}
    \min_{(\vv,\NN)\in X_*} J_*(\vv,\NN;\ff)
  \end{align*}
  has a unique minimizer $(\uu,\MM)\in X_*$ with $J_*(\uu,\MM;\ff)=0$ and the pair $(\uu,p):=(\uu,-d^{-1}t\tr\MM)$ solves problem~\eqref{eq:model}.

  Let $X_{*,h}\subset X_*$ denote a closed subspace. The minimization problem
  \begin{align*}
    \min_{(\vv_h,\NN_h)\in X_{*,h}} J_*(\vv_h,\NN_h;\ff)
  \end{align*}
  has a unique minimizer $(\uu_h,\MM_h)\in X_{*,h}$. Moreover, quasi-optimality
  \begin{align*}
    \enorm{(\uu-\uu_h,\MM-\MM_h)}_t \leq C \inf_{(\vv_h,\NN_h)\in X_{*,h}} 
    \enorm{(\uu-\vv_h,\MM-\NN_h)}_t
  \end{align*}
  holds, and the least-squares functional induces an efficient and reliable a posteriori error estimator, i.e., 
  \begin{align*}
    C^{-1} \enorm{(\uu-\vv_h,\MM-\NN_h)}_t^2 \leq J_*(\vv_h,\NN_h;\ff) 
    \leq C \enorm{(\uu-\vv_h,\MM-\NN_h)}_t^2 
  \end{align*}
  for any $(\vv_h,\NN_h)\in X_{*,h}$. 
  The constant $C>0$ depends on $\Omega$ but is independent of $t\in(0,1]$.
\end{corollary}

We point out that the condition $\int_\Omega \tr\MM\,\di\xx = 0$ is included in the space $X_*$. 
In practice, such a side constraint is often realized using Lagrangian multipliers which leads to indefinite algebraic
systems. 
In order to preserve the advantageous properties of LSFEMs we therefore consider an alternative least-squares functional whose discretization leads to a symmetric positive definite system.
For the new functional we simply add the zero average condition in a least-squares fashion, i.e.,
\begin{align*}
  J(\uu,\MM;\ff) := J_*(\uu,\MM;\ff) + t^2\norm{\Pi^0 \tr\MM}{}^2.
\end{align*}
Here, $\Pi^0q = |\Omega|^{-1}\int q\,d\xx$ is the $L^2(\Omega)$ projection on to constant functions.
\begin{theorem}\label{thm:lsq:J}
  There exists a constant $C>0$ which depends on $\Omega$ but is independent of $t\in(0,1]$ such that
  \begin{align}
    C^{-1} \enorm{(\uu,\MM)}_t^2 \leq J(\uu,\MM;0) \leq
    C \enorm{(\uu,\MM)}_t^2
    \quad\text{for all } (\uu,\MM)\in X.
  \end{align}
\end{theorem}
\begin{proof}
  The upper bound follows as in the proof of Theorem~\ref{thm:lsq} and the fact that $t^2\norm{\Pi^0 \tr\MM}{}\leq t^2\norm{\tr\MM}{}$.

  The lower bound can be seen by using $(1-\Pi^0)\tr\MM \in L_*^2(\Omega)$ so that $\widetilde\MM := \MM-d^{-1}\Pi^0\tr(\MM)\II$ satisfies $\tr\widetilde\MM = (1-\Pi^0)\tr\MM\in L_*^2(\Omega)$.
  The triangle inequality, and Theorem~\ref{thm:lsq} further prove that
  \begin{align*}
    \enorm{(\uu,\MM)}_t^2 &\lesssim \enorm{(\uu,\widetilde\MM)}_t^2 + t^2\norm{\Pi^0 \tr\MM}{}^2
    \lesssim J_*(\uu,\widetilde\MM;0) + t^2\norm{\Pi^0\tr\MM}{}^2
    \\
    &= \norm{-t\Div\widetilde\MM}{}^2 + \norm{\Dev\widetilde\MM -t\Grad\uu}{}^2 + \norm{\div\uu}{}^2 + t^2\norm{\Pi^0\tr\MM}{}^2
    \\
    &= \norm{-t\Div\MM}{}^2 + \norm{\Dev\MM -t\Grad\uu}{}^2 + \norm{\div\uu}{}^2 + t^2\norm{\Pi^0\tr\MM}{}^2 
    \\
    &= J(\uu,\MM;0).
  \end{align*}
  This finishes the proof.
\end{proof}

The following result is an immediate consequence of Theorem~\ref{thm:lsq:J} and the theory on least-squares finite element methods.
\begin{corollary}\label{cor:lsq:J}
  Let $\ff\in\LL^2(\Omega)$ be given. The minimization problem
  \begin{align*}
    \min_{(\vv,\NN)\in X} J(\vv,\NN;\ff)
  \end{align*}
  has a unique minimizer $(\uu,\MM)\in X$ with $J(\uu,\MM;\ff)=0$ and the pair $(\uu,p):=(\uu,-d^{-1}t\tr\MM)$ solves problem~\eqref{eq:model}.

  Let $X_{h}\subset X$ denote a closed subspace. The minimization problem
  \begin{align*}
    \min_{(\vv_h,\NN_h)\in X_{h}} J(\vv_h,\NN_h;\ff)
  \end{align*}
  has a unique minimizer $(\uu_h,\MM_h)\in X_{h}$. Moreover, quasi-optimality
  \begin{align*}
    \enorm{(\uu-\uu_h,\MM-\MM_h)}_t \leq C \inf_{(\vv_h,\NN_h)\in X_{h}} 
    \enorm{(\uu-\vv_h,\MM-\NN_h)}_t
  \end{align*}
  holds, and the least-squares functional induces an efficient and reliable a posteriori error estimator, i.e., 
  \begin{align*}
    C^{-1} \enorm{(\uu-\vv_h,\MM-\NN_h)}_t^2 \leq J(\vv_h,\NN_h;\ff) 
    \leq C \enorm{(\uu-\vv_h,\MM-\NN_h)}_t^2 
  \end{align*}
  for any $(\vv_h,\NN_h)\in X_h$. 
  The constant $C>0$ depends on $\Omega$ but is independent of $t\in(0,1]$.
\end{corollary}

Under some weak assumptions on the spaces which are satisfied for $X$ and standard discretization spaces we show that
the functionals $J_*$ and $J$ have the same minimizer on the discrete level.
\begin{theorem}\label{thm:equivalence}
  Let $\ff\in\LL^2(\Omega)$. Let $X_h\subseteq X$ denote a closed subspace 
  such that $X_{*,h} := X_h\cap X_*$ is non-empty. If 
  \begin{align}
    (0,\II)\in X_h,
  \end{align}
  then the two minimization problems
  \begin{align}
    &\min_{(\vv_h,\NN_h)\in X_{*,h}} J_*(\vv_h,\NN_h;\ff), \label{eq:minProb:a} \\
    &\min_{(\vv_h,\NN_h)\in X_{h}} J(\vv_h,\NN_h;\ff), \label{eq:minProb:b}
  \end{align}
  have the same (unique) minimizer $(\uu_h,\MM_h)\in X_{*,h}$.
\end{theorem}
\begin{proof}
  In the proof we utilize the characterization of the minimizers by the Euler--Lagrange equations.
  For both cases we have the same right-hand side
  \begin{align*}
    L(\vv,\NN) := \ip{\ff}{-t\Div\NN+\vv} \quad\forall (\vv,\NN)\in X.
  \end{align*}

  We define the bilinear forms
  \begin{align*}
    b_*(\uu,\MM;\vv,\NN) &:= 
    \ip{-t\Div\MM+\uu}{-t\Div\NN+\vv} \\ &\qquad + \ip{\Dev\MM-t\Grad\uu}{\Dev\NN-t\Grad\vv}
    + \ip{\div\uu}{\div\vv} \\
    b(\uu,\MM;\vv,\NN) &:= b_*(\uu,\MM;\vv,\NN) + t^2\ip{\Pi^0\tr\MM}{\Pi^0\tr\NN}
  \end{align*}
  for all $(\uu,\MM),(\vv,\NN)\in X$.
  \newline
  The Euler--Lagrange equations for~\eqref{eq:minProb:a} read: Find $(\uu_h,\MM_h)\in X_{*,h}$:
  \begin{align}\label{eq:eulerlagrange:a}
    b_*(\uu_h,\MM_h;\vv,\NN) = L(\vv,\NN) \quad\text{for all } (\vv,\NN) \in X_{*,h}.
  \end{align}
  The Euler--Lagrange equations for~\eqref{eq:minProb:b} read: Find $(\uu_h,\MM_h)\in X_{h}$:
  \begin{align}\label{eq:eulerlagrange:b}
    b(\uu_h,\MM_h;\vv,\NN) = L(\vv,\NN) \quad\text{for all } (\vv,\NN) \in X_{h}.
  \end{align}

  \noindent
  \textbf{The solution of~\eqref{eq:eulerlagrange:a} solves~\eqref{eq:eulerlagrange:b}.}
  Let $(\uu_h,\MM_h)\in X_{*,h}$ denote the solution of~\eqref{eq:eulerlagrange:a}. 
  Since $b_*(\uu_h,\MM_h;\vv,\NN) = b(\uu_h,\MM_h;\vv,\NN)$ for all $(\vv,\NN)\in X_{*,h}$ it suffices to verify the identity
  $b(\uu_h,\MM_h;0,\II) = L(0,\II) = 0$.
  Using $\Div\II = 0$, $\Dev\II = 0$, and $\tr\MM_h\in L_*^2(\Omega)$, we see that $b_*(\uu,\MM_h;0,\II) = 0$, and, thus,
  \begin{align*}
    b(\uu_h,\MM_h;0,\II) = b_*(\uu_h,\MM_h;0,\II) + t^2\ip{\tr\MM_h}{\tr\II} = t^2\ip{\tr\MM_h}d = 0.
  \end{align*}

  \noindent
  \textbf{The solution of~\eqref{eq:eulerlagrange:b} solves~\eqref{eq:eulerlagrange:a}.}
  Let $(\uu_h,\MM_h)\in X_h$ denote the solution of~\eqref{eq:eulerlagrange:b}. It is sufficient to prove that $\Pi^0\tr\MM_h=0$. 
  This follows from testing with $(0,\II)\in X_h$ which yields with the similar argumentation as in the previous step that
  \begin{align*}
    t^2\ip{\tr\MM_h}{\tr\II} = 0
  \end{align*}
  or equivalently $\Pi^0\tr\MM_h=0$. 
\end{proof}

Some remarks are in order.
\begin{remark}
  We note that the additional term $t^2\norm{\Pi^0\tr\MM}{}^2$ in the functional $J$ is a non-local term. 
  However, using an iterative solver this term can efficiently be implemented since it is a rank-$1$ term.

  The condition $(0,\II)\in X_h$ is satisfied for generic approximation spaces, see, e.g., Section~\ref{sec:discretization} below.
\end{remark}
\begin{remark}
  Approximations of the pressure can be recovered from the pseudostress variable: Let $(u_h,\MM_h)\in X_h$ denote the solution of~\eqref{eq:eulerlagrange:b}. Recall that $p = -d^{-1}t\tr\MM$ and set $p_h = -d^{-1}t\tr\MM_h$. The quasi-optimality of the least-squares methods proves
  \begin{align*}
    \norm{p-p_h}{} \leq d^{-1} t\norm{\tr(\MM-\MM_h)}{} &\lesssim \enorm{(\uu-\uu_h,\MM-\MM_h)}_t 
    \\
    &\lesssim \min_{(\vv_h,\NN_h)\in X_h} \enorm{(\uu-\vv_h,\MM-\NN_h)}_t.
  \end{align*}
\end{remark}

\section{Discretization}\label{sec:discretization}
We stress that least-squares methods are well-defined for any choice of conforming discrete spaces.
However, for simplicity we restrict our presentation to simplicial meshes and standard approximation spaces found in the
literature.

\subsection{Meshes and approximation spaces}
Let $\TT$ denote a regular mesh of $\Omega$ into (relatively open) simplices $T$, i.e.,
\begin{align*}
  \overline\Omega = \bigcup_{T\in\TT} \overline T.
\end{align*}
We assume that $\TT$ is $\kappa$-shape regular, i.e.,
\begin{align*}
  \max_{T\in\TT} \frac{\diam(T)^d}{|T|} \leq \kappa <\infty.
\end{align*}

The mesh-width function $h_\TT\in L^\infty(\Omega)$ is defined by
\begin{align*}
  h_\TT|_T := h_T := \diam(T) \quad\text{for all }T\in\TT.
\end{align*}
For a fixed mesh we set $h:=\max_{T\in\TT} h_T$.

Let $\PP^k(T)$ denote the set of polynomials on an element $T$ of degree less or equal to $k\in\N_0$. We define
\begin{align*}
  \PP^k(\TT) := \set{v\in L^2(\Omega)}{v|_T \in \PP^k(T) \text{ for all }T\in\TT}.
\end{align*}
Furthermore, we need the spaces
\begin{align*}
  \cS^{k+1}(\TT) &:= \PP^{k+1}(\TT)\cap H^1(\Omega), \quad
  \cS_0^{k+1}(\TT) := \cS^{k+1}(\TT)\cap H_0^1(\Omega), \\
  \cSS_0^{k+1}(\TT) &:= \cS_0^{k+1}(\TT)^d.
\end{align*}

The local Raviart--Thomas $\RT^k(T)$ space is given by
\begin{align*}
  \RT^k(T) := \PP^k(T)^d + \xx \PP^k(T)
\end{align*}
and the global one by
\begin{align*}
  \RT^k(\Omega) := \set{\vv\in \LL^2(\Omega)}{\div\vv\in L^2(\Omega) \text{ and }
    \vv|_T\in \RT^k(T) \text{ for all }T\in\TT}.
  \end{align*}
We will use the space
\begin{align*}
  \boldsymbol\RT^k(\TT) := \set{\MM\in \HDivset\Omega}{\MM_{j,\cdot} \in \RT^k(\TT) \text{ for }
  j=1,\dots d}.
\end{align*}
Here, $\MM_{j,\cdot}$ denotes the $j$-th row of $\MM$.

Moreover, in the analysis below we use the following interpolation operators:
\begin{itemize}
  \item $\Pi_{L^2}^k \colon L^2(\Omega)\to \PP^k(\TT)$ the $L^2(\Omega)$ projection with
    \begin{align*}
      \norm{q-\Pi_{L^2}^kq}{} \lesssim h^{k+1}\norm{q}{H^{k+1}(\Omega)}, 
    \end{align*}
  \item $\Pi_{\RT}^k\colon \boldsymbol\HH^1(\Omega)\to \boldsymbol\RT^k(\TT)$ the Raviart--Thomas projector (applied to
    each row of the matrix) with
    \begin{align*}
      \norm{\NN-\Pi_{\RT}^k\NN}{} &\lesssim h^{k+1}\norm{\NN}{H^{k+1}(\Omega)}
      \\
      \Div(\NN-\Pi_{\RT}^k\NN) &= (1-\Pi_{L^2}^k)\NN,
    \end{align*}
\end{itemize}
The constants involved in the estimates depend only on the shape-regularity of $\TT$ and $k\in\N_0$.

\subsection{Locking effect for a standard discretization}
We investigate the discrete space
\begin{align*}
  X_h := \cSS_0^{k+1}(\TT) \times \boldsymbol\RT^k(\TT).
\end{align*}
It is easy to check that $(0,\II)\in X_h$. Thus, the assumptions of Theorem~\ref{thm:equivalence}
are satisfied. In particular, this means we can restrict our investigations to the functional $J$ (Theorem~\ref{thm:equivalence}).

Let $\ff\in \LL^2(\Omega)$ be given and let $(\uu,\MM)\in X$, $(\uu_h,\MM_h)\in X_h$ denote the solution of the
minimization problems
\begin{align*}
  \min_{(\vv,\MM)\in X} J(\vv,\NN;\ff) \quad\text{and}\quad
  \min_{(\vv_h,\MM_h)\in X_h} J(\vv_h,\NN_h;\ff)
\end{align*}
which are unique due to Theorem~\ref{thm:lsq:J}.
Recall that
\begin{align*}
  \enorm{(\uu-\uu_h,\MM-\MM_h)}_t \leq C \inf_{(\vv_h,\NN_h)\in X_{h}} 
  \enorm{(\uu-\vv_h,\MM-\NN_h)}_t.
\end{align*}
It is clear from the approximation properties of the involved spaces and the definition of the norms that the
best-approximation error, i.e., the right-hand side of the latter estimate, is $\OO(h^{k+1})$ provided that the
solution $(\uu,\MM)$ is sufficiently smooth.
However, such estimates depend on higher order Sobolev norms of the solution $(\uu,\MM)$ which in general are not independent of the singular perturbation parameter $t$. 
As we show below in Example~\ref{ex:convrates}, the space $X_h$ may lead to a locking effect that is also seen in numerical experiments, see Section~\ref{sec:num:comp}.
One reason is that a projection of $\MM$ with $\Dev\MM=0$ onto the Raviart--Thomas space in general leads to an element
$\MM_h$ with $\Dev\MM_h\neq 0$.
The following example with a simple manufactured solution makes this statement more precise.
\begin{example}\label{ex:convrates}
  Let $\Omega = (0,1)^2$ and set $\uu = 0$, $p(x,y) = x^2-\tfrac13$, $\MM = -t^{-1}p\II$. 
  Note that this choice is the unique solution of~\eqref{eq:model} with $\ff = \nabla p$.
  To estimate the best-approximation error with respect to $\enorm\cdot_t$ we set $\vv_h=0$ and $\NN_h = \Pi_{\RT}^k \MM$. Then,
  \begin{align*}
    \enorm{(\uu-\vv_h,\MM-\NN_h)}_t^2 &= t^2\norm{\Div(\MM-\NN_h)}{}^2 
    + \norm{\Dev(\MM-\NN_h)}{}^2 \\
    &\qquad+ t^2\norm{\tr(\MM-\NN_h)}{}^2 .
  \end{align*}
  The commutativity property of the Raviart--Thomas projection proves that the first term satisfies the estimate
  \begin{align*}
    t^2\norm{\Div(\MM-\NN_h)}{}^2 \lesssim h^{2(k+1)}\norm{p}{H^{k+2}(\Omega)}^2.
  \end{align*}
  The last term is estimated by
  \begin{align*}
    t^2\norm{\tr(\MM-\NN_h)}{}^2 = \norm{\tr((1-\Pi_{\RT}^k)p\II)}{}^2 \lesssim h^{2(k+1)}\norm{p}{H^{k+1}(\Omega)}^2.
  \end{align*}
  Note that the last estimate is independent of $t\in(0,1]$.
  It remains to analyze the deviatoric part. We have that $\NN_h = \Pi_{\RT}^k(-t^{-1}p\II)$.
  It can easily be verified that $\Dev\NN_h\neq 0$. Therefore,
  \begin{align*}
    \norm{\Dev(\MM-\NN_h)}{}^2 = t^{-2} \norm{\Dev(1-\Pi_{\RT}^k)p\II}{}^2.
  \end{align*}
  Since $p$ is completely independent of $t$ we can not get rid of $t^{-2}$.
  This can lead to locking effects which are seen in numerical experiments, cf. Section~\ref{sec:num:comp}.
\end{example}

\subsection{Augmented discretization space}\label{sec:regularities}
In order to avoid such effects described in the last section, we augment the discrete space with deviatoric free elements.
It is straightforward to verify that if we want elementwise polynomial spaces which are deviatoric free and subspaces of
$\HDivset\Omega$ such elements take the form
\begin{align*}
  \eta\II \quad\text{with } \eta\in \cS^{k+1}(\TT).
\end{align*}
Since $\II\in\boldsymbol\RT^k(\TT)$ it suffices to add zero average functions.
With $\cS_*^{k+1}(\TT)\II := \set{\eta\II}{\eta\in \cS^{k+1}(\TT)\cap L_*^2(\Omega)}$ we define the space
\begin{align*}
  X_h^+ := \cSS_0^{k+1}(\TT) \times \left(\boldsymbol\RT^k(\TT)+\cS_*^{k+1}(\TT)\II\right).
\end{align*}
We stress that $\boldsymbol\RT^k(\TT)\cap\big(\cS_*^{k+1}(\TT)\setminus\cS_*^{k}(\TT)\big)\II = \{0\}$.

Concerning Example~\ref{ex:convrates} the space $X_h^+$ leads to locking free approximations:
\begin{example}\label{ex:convrates2}
  Consider the setup of Example~\ref{ex:convrates}. We estimate the best-approximation error with respect to the norm $\enorm\cdot_t$ and the space $X_h^+$. 
  To that end we choose $\vv_h = 0$ and $\NN_h = \Pi_{\RT}^k\Dev\MM + \Pi_{\cS}^{k+1}\tfrac{p}t\II$. Since $\Dev\MM = t\Grad\uu =0$ we have that $\NN_h = \Pi_{\cS}^{k+1}\tfrac{p}t\II$ and $\Dev\NN_h = 0$.
  Thus, we end up with
  \begin{align*}
    \enorm{(\uu-\vv_h,\MM-\NN_h)}_t^2 &= t^2\norm{\Div(1-\Pi_{\cS}^{k+1})\tfrac{p}t\II}{}^2 + t^2\norm{\tr(1-\Pi_{\cS}^{k+1})\tfrac{p}t\II}{}^2 
    \\
    &\lesssim \norm{\nabla(1-\Pi_{\cS}^{k+1})p}{}^2 + \norm{(1-\Pi_{\cS}^{k+1})p}{}^2 
    \\&\lesssim h^{2(k+1)}\norm{p}{H^{k+2}(\Omega)}^2,
  \end{align*}
  where the involved constants are independent of $t\in(0,1]$.
\end{example}

\section{Numerical Examples}\label{sec:num}
In this section we present different numerical examples for $d=2$ and $k=0$.
Note that on uniform shape-regular meshes we have that $h\eqsim \diam(T)$ for all $T\in \TT$ and $\#\TT \eqsim h^{-2}$.
Therefore, the optimal convergence in the energy norm is $\OO(h) = \OO( (\#\TT)^{-1/2}) = \OO( (\dim X_h)^{-1/2})$.

For the experiments from Section~\ref{sec:num:layer} and Section~\ref{sec:num:Lshape} we also employ a basic adaptive algorithm with the four steps
\begin{align*}
  \boxed{\mathrm{Solve}} \Longrightarrow \boxed{\mathrm{Estimate}} \Longrightarrow \boxed{\mathrm{Mark}} \Longrightarrow \boxed{\mathrm{Refine}}.
\end{align*}
As mesh-refinement we use the newest-vertex bisection algorithm. Elements are marked for refinement using D\"orfler's bulk criterion: Find a minimal set $\mathcal{M}\subset\TT$ such that
\begin{align*}
  \theta \eta^2 \leq \sum_{T\in\mathcal{M}} \eta(T)^2. 
\end{align*}
Throughout, we use $\theta = \tfrac14$ as marking parameter.
Here, $\eta$ denotes the a posteriori error estimator given by the natural error estimator of the least-squares method, i.e., 
$J(\uu_h,\MM_h;\ff) =: \eta^2 = \sum_{T\in\TT} \eta(T)^2$ and the local error indicators are given on any $T\in\TT$ by
\begin{align*}
  \eta(T)^2 := \norm{-t\Div\MM_h+\uu_h-\ff}{T}^2 + \norm{\Dev\MM_h-t\Grad\uu_h}T^2 + \norm{\div\uu_h}T^2.
\end{align*}

Furthermore, we use the abbreviations
\begin{alignat*}{2}
  \err(\uu_h)^2 &:= \norm{\uu-\uu_h}{t}^2 &\,=\,& \norm{\uu-\uu_h}{}^2 + t^2\norm{\Grad(\uu-\uu_h)}{}^2 
  \\& &\,&\qquad + \norm{\div\uu_h}{}^2, \\
  \err(\MM_h)^2 &:= \norm{\MM-\MM_h}{t}^2 &=& \norm{\Dev(\MM-\MM_h)}{}^2 + t^2 \norm{\tr(\MM-\MM_h)}{}^2 
  \\ & &\,& \qquad +t^2\norm{\div(\MM-\MM_h)}{}^2
\end{alignat*}
for the errors between the exact solution $(\uu,\MM)$ and its least-squares approximation $(\uu_h,\MM_h)$.

\subsection{Minimization over $X_h$ and $X_h^+$}\label{sec:num:comp}
We consider the problem from Example~\ref{ex:convrates} with $\Omega = (0,1)^2$, $\uu = 0$, $p(x,y) = x^2-\tfrac13$ and $\MM =
-t^{-1}p\II$. Then, the right-hand side is given by $\ff = \nabla p = (2x,0)^\top$.
We compare the minimizers of the functional $J(\vv,\NN;\ff)$ over the space $X_h$ and $X_h^+$ for different values of
$t\in(0,1]$.
Figure~\ref{fig:min} shows the values of $J(\uu_h,\MM_h;\ff)^{1/2}$ where $(\uu_h,\MM_h)$ is either the minimizer in
$X_h$ (left plot) or the minimizer in $X_h^+$ (right plot). Recall that (Corollary~\ref{cor:lsq:J})
\begin{align*}
  J(\uu_h,\MM_h;\ff) \simeq \enorm{(\uu-\uu_h,\MM-\MM_h)}_t^2.
\end{align*}
One observes that for smaller values of $t$ a locking effect occurs when minimizing over the set $X_h$ whereas
minimization over $X_h^+$ leads to optimal convergence rates.
This fits the observations from Examples~\ref{ex:convrates} and~\ref{ex:convrates2}.

\begin{figure}
  \begin{tikzpicture}
\begin{loglogaxis}[
    title={Minimization over $X_h$},
width=0.49\textwidth,
cycle list/Dark2-6,
cycle multiindex* list={
mark list*\nextlist
Dark2-6\nextlist},
every axis plot/.append style={ultra thick},
xlabel={degrees of freedom},
grid=major,
legend entries={\tiny $t=1$,\tiny $t=10^{-1}$, \tiny $t=10^{-2}$, \tiny $t=10^{-3}$},
legend pos=south west,
]
\addplot table [x=dof,y=est] {data/Example1_1000.dat};
\addplot table [x=dof,y=est] {data/Example1_100.dat};
\addplot table [x=dof,y=est] {data/Example1_10.dat};
\addplot table [x=dof,y=est] {data/Example1_1.dat};

\logLogSlopeTriangle{0.9}{0.2}{0.37}{0.5}{black}{{\tiny $\tfrac12$}};
\end{loglogaxis}
\end{tikzpicture}
\begin{tikzpicture}
\begin{loglogaxis}[
    title={Minimization over $X_h^+$},
width=0.49\textwidth,
cycle list/Dark2-6,
cycle multiindex* list={
mark list*\nextlist
Dark2-6\nextlist},
every axis plot/.append style={ultra thick},
xlabel={degrees of freedom},
grid=major,
legend entries={\tiny $t=1$,\tiny $t=10^{-1}$, \tiny $t=10^{-2}$, \tiny $t=10^{-3}$},
legend pos=south west,
]
\addplot table [x=dof,y=est] {data/Example1_Alt_1000.dat};
\addplot table [x=dof,y=est] {data/Example1_Alt_100.dat};
\addplot table [x=dof,y=est] {data/Example1_Alt_10.dat};
\addplot table [x=dof,y=est] {data/Example1_Alt_1.dat};

\logLogSlopeTriangle{0.9}{0.2}{0.2}{0.5}{black}{{\tiny $\tfrac12$}};
\end{loglogaxis}
\end{tikzpicture}
  \caption{Estimator $J(\uu_h,\MM_h,p_h;\ff)$ for the problem described in Section~\ref{sec:num:comp}.}
  \label{fig:min}
\end{figure}
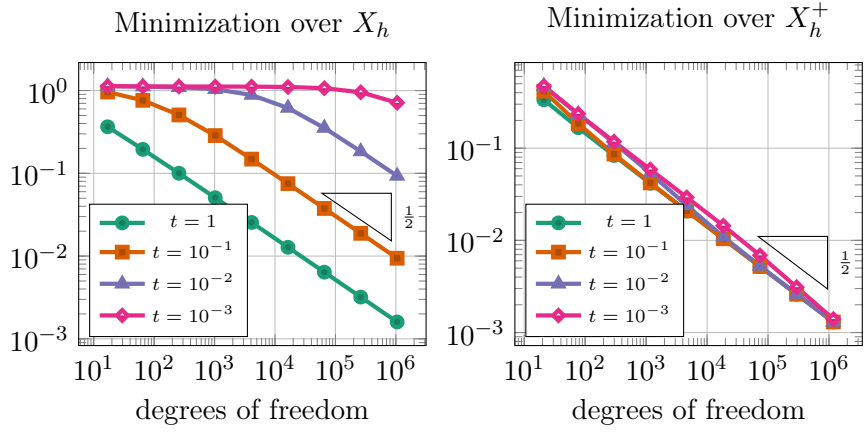

\subsection{Example with known solution}\label{sec:num:layer}
We consider $\Omega = (0,1)^2$ and the manufactured solution
\begin{align*}
  p(x,y) &= 0, \\
  \uu_1(x,y) &= \frac{1+e^{1/t}-e^{y/t}-e^{(1-y)/t}}{1+e^{1/t}},  \\
  \uu_2(x,y) &= 0.
\end{align*}
Then, $\ff = -t^2\boldsymbol{\Delta}\uu+\uu+\nabla p = (1,0)^\top$. 
We note that $\uu$ does not satisfy homogengeous boundary conditions at $x=0$ and $x=1$ and has a prototypical boundary layer close to $y=0$ and $y=1$.
A similar example has been considered in~\cite[Sec.~4]{HannukainenJuntunenStenberg13}. It is Poiseuille flow in the $x$-direction.
We include boundary conditions in a canonical way by defining $\uu_{h,\Gamma} \in \cS^1(\TT)^2$ as $\uu_{h,\Gamma}(\zz) = \uu(\zz)$ at all boundary vertices and $\uu_{h,\Gamma}(\zz) = 0$ for all interior vertices. Then, writing $\uu_h = \uu_{h,0} + \uu_{h,\Gamma}\in \cS^1(\TT)^2$ we solve for $\uu_{h,0}$ which satisfies the homogeneous boundary conditions.
Figure~\ref{fig:example2} shows the results for $t=5\cdot10^{-2}, 5\cdot10^{-3}$ on a sequence of uniformly as well as adaptively refined meshes. 
The adaptive algorithm seems to detect the boundary layer and its preasymptotic range is left earlier than with the uniformly refined meshes. This effect depends on $t$ and is seen more clearly with smaller $t$.

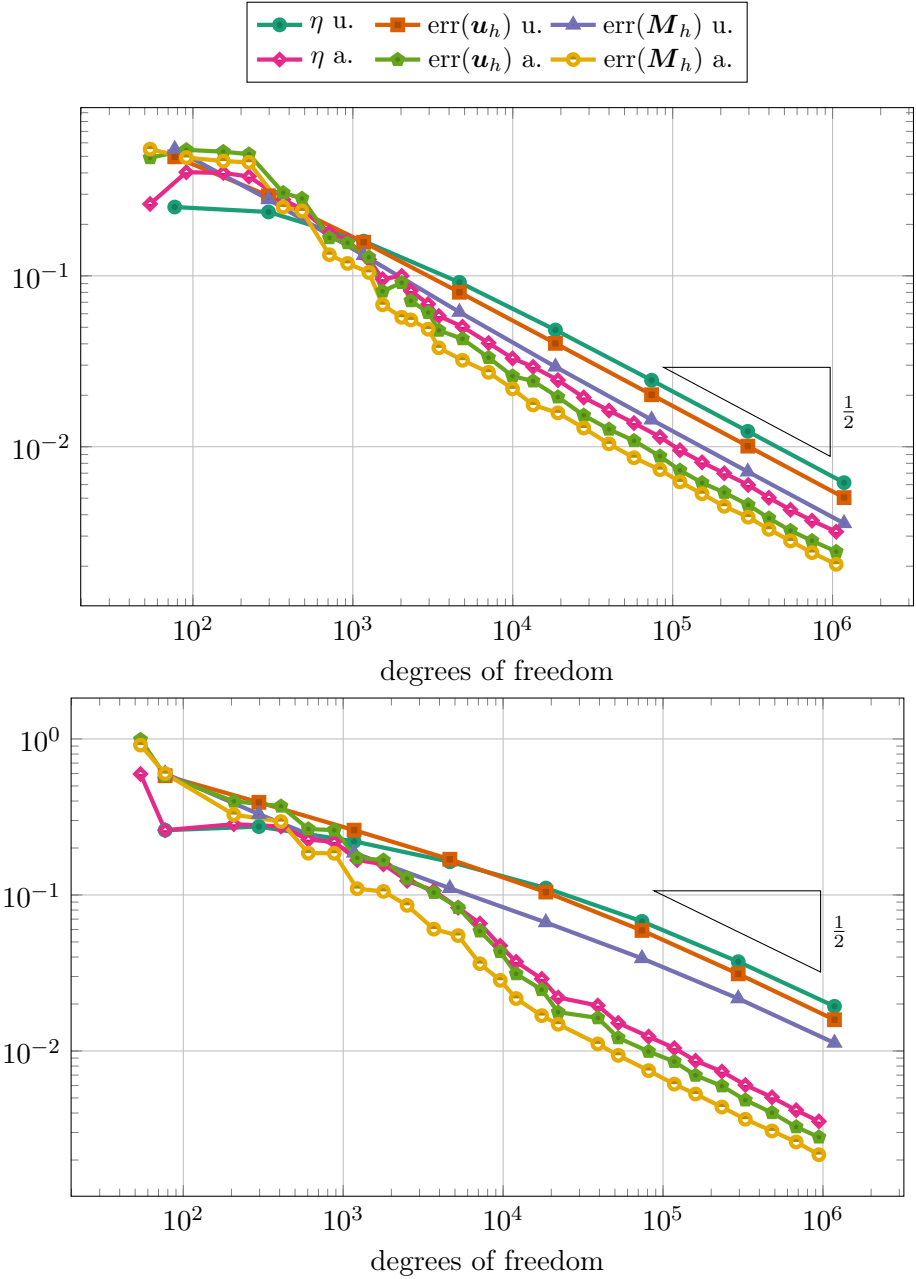
\begin{figure}
  \begin{tikzpicture}
\begin{loglogaxis}[
width=\textwidth,
height=0.4\textheight,
cycle list/Dark2-6,
cycle multiindex* list={
mark list*\nextlist
Dark2-6\nextlist},
every axis plot/.append style={ultra thick},
xlabel={degrees of freedom},
grid=major,
legend entries={\small $\eta$ u.,\small $\err(\uu_h)$ u., \small $\err(\MM_h)$ u.,\small $\eta$ a.,\small $\err(\uu_h)$ a., \small $\err(\MM_h)$ a.},
legend style={at={(0.5,1.05)},anchor=south},
legend columns=3, 
        legend style={
            /tikz/column 2/.style={
                column sep=5pt,
            }},
]
\addplot table [x=dof,y=est] {data/Example2_Unif_5e-02.dat};
\addplot table [x=dof,y=errU] {data/Example2_Unif_5e-02.dat};
\addplot table [x=dof,y=errM] {data/Example2_Unif_5e-02.dat};
\addplot table [x=dof,y=est] {data/Example2_Adap_5e-02.dat};
\addplot table [x=dof,y=errU] {data/Example2_Adap_5e-02.dat};
\addplot table [x=dof,y=errM] {data/Example2_Adap_5e-02.dat};

\logLogSlopeTriangle{0.9}{0.2}{0.3}{0.5}{black}{{\small $\tfrac12$}};
\end{loglogaxis}
\end{tikzpicture}
\begin{tikzpicture}
\begin{loglogaxis}[
width=\textwidth,
height=0.4\textheight,
cycle list/Dark2-6,
cycle multiindex* list={
mark list*\nextlist
Dark2-6\nextlist},
every axis plot/.append style={ultra thick},
xlabel={degrees of freedom},
grid=major,
]
\addplot table [x=dof,y=est] {data/Example2_Unif_5e-03.dat};
\addplot table [x=dof,y=errU] {data/Example2_Unif_5e-03.dat};
\addplot table [x=dof,y=errM] {data/Example2_Unif_5e-03.dat};
\addplot table [x=dof,y=est] {data/Example2_Adap_5e-03.dat};
\addplot table [x=dof,y=errU] {data/Example2_Adap_5e-03.dat};
\addplot table [x=dof,y=errM] {data/Example2_Adap_5e-03.dat};

\logLogSlopeTriangle{0.9}{0.2}{0.45}{0.5}{black}{{\small $\tfrac12$}};
\end{loglogaxis}
\end{tikzpicture}
  \caption{Estimator $\eta$ and errors $\norm{\uu-\uu_h}t$, $\norm{\MM-\MM_h}{t}$ for the problem described in Section~\ref{sec:num:layer} on a sequence of uniform (u.) resp. adaptively (a.) refined meshes.
  The top plot corresponds to the results for $t=5\cdot10^{-2}$ and the bottom plot to $t=5\cdot 10^{-3}$.}
  \label{fig:example2}
\end{figure}

\subsection{Example in non-convex domain}\label{sec:num:Lshape}
We consider the L-shaped domain $\Omega = (-1,1)^2\setminus[-1,0]^2$ and 
\begin{align*}
  \ff(x,y) = \begin{pmatrix} xy \\ e^x \end{pmatrix}.
\end{align*}
The explicit representation for the solution $(\uu,p)$ of~\eqref{eq:model} is not known. 
Figure~\ref{fig:Lshape} shows the results obtained for a sequence of uniformly and adaptively refined meshes. 
The quantities $\eta:=J(\uu_h,\MM_h;\ff)^{1/2}$ as well as $\norm{\div\uu_h}{}$ are presented. One observes that the adaptive algorithm yields superior results when $t$ becomes smaller.

\begin{figure}
  \begin{tikzpicture}
\begin{loglogaxis}[
width=0.49\textwidth,
cycle list/Dark2-6,
cycle multiindex* list={
mark list*\nextlist
Dark2-6\nextlist},
every axis plot/.append style={ultra thick},
xlabel={degrees of freedom},
grid=major,
legend entries={\tiny $\eta$ u.,\tiny $\norm{\div\uu_h}{}$ u.,\tiny $\eta$ a.,\tiny $\norm{\div\uu_h}{}$ a.},
legend style={at={(0.5,1.05)},anchor=south},
legend columns=2, 
        legend style={
            /tikz/column 2/.style={
                column sep=5pt,
            }},
]
\addplot table [x=dof,y=est] {data/Example5_Unif_1e+00.dat};
\addplot table [x=dof,y=normDivU] {data/Example5_Unif_1e+00.dat};
\addplot table [x=dof,y=est] {data/Example5_Adap_1e+00.dat};
\addplot table [x=dof,y=normDivU] {data/Example5_Adap_1e+00.dat};

\end{loglogaxis}
\end{tikzpicture}
\begin{tikzpicture}
\begin{loglogaxis}[
width=0.49\textwidth,
cycle list/Dark2-6,
cycle multiindex* list={
mark list*\nextlist
Dark2-6\nextlist},
every axis plot/.append style={ultra thick},
xlabel={degrees of freedom},
grid=major,
legend entries={\tiny $\eta$ u.,\tiny $\norm{\div\uu_h}{}$ u.,\tiny $\eta$ a.,\tiny $\norm{\div\uu_h}{}$ a.},
legend style={at={(0.5,1.05)},anchor=south},
legend columns=2, 
        legend style={
            /tikz/column 2/.style={
                column sep=5pt,
            }},
]
\addplot table [x=dof,y=est] {data/Example5_Unif_1e-01.dat};
\addplot table [x=dof,y=normDivU] {data/Example5_Unif_1e-01.dat};
\addplot table [x=dof,y=est] {data/Example5_Adap_1e-01.dat};
\addplot table [x=dof,y=normDivU] {data/Example5_Adap_1e-01.dat};

\end{loglogaxis}
\end{tikzpicture}
\begin{tikzpicture}
\begin{loglogaxis}[
width=0.49\textwidth,
cycle list/Dark2-6,
cycle multiindex* list={
mark list*\nextlist
Dark2-6\nextlist},
every axis plot/.append style={ultra thick},
xlabel={degrees of freedom},
grid=major,
legend entries={\tiny $\eta$ u.,\tiny $\norm{\div\uu_h}{}$ u.,\tiny $\eta$ a.,\tiny $\norm{\div\uu_h}{}$ a.},
legend style={at={(0.5,1.05)},anchor=south},
legend columns=2, 
        legend style={
            /tikz/column 2/.style={
                column sep=5pt,
            }},
]
\addplot table [x=dof,y=est] {data/Example5_Unif_1e-02.dat};
\addplot table [x=dof,y=normDivU] {data/Example5_Unif_1e-02.dat};
\addplot table [x=dof,y=est] {data/Example5_Adap_1e-02.dat};
\addplot table [x=dof,y=normDivU] {data/Example5_Adap_1e-02.dat};

\end{loglogaxis}
\end{tikzpicture}
\begin{tikzpicture}
\begin{loglogaxis}[
width=0.49\textwidth,
cycle list/Dark2-6,
cycle multiindex* list={
mark list*\nextlist
Dark2-6\nextlist},
every axis plot/.append style={ultra thick},
xlabel={degrees of freedom},
grid=major,
legend entries={\tiny $\eta$ u.,\tiny $\norm{\div\uu_h}{}$ u.,\tiny $\eta$ a.,\tiny $\norm{\div\uu_h}{}$ a.},
legend style={at={(0.5,1.05)},anchor=south},
legend columns=2, 
        legend style={
            /tikz/column 2/.style={
                column sep=5pt,
            }},
]
\addplot table [x=dof,y=est] {data/Example5_Unif_1e-03.dat};
\addplot table [x=dof,y=normDivU] {data/Example5_Unif_1e-03.dat};
\addplot table [x=dof,y=est] {data/Example5_Adap_1e-03.dat};
\addplot table [x=dof,y=normDivU] {data/Example5_Adap_1e-03.dat};

\end{loglogaxis}
\end{tikzpicture}
  \caption{Estimator $\eta$ and error $\norm{\div\uu_h}{}$ on a sequence of uniform (u.) resp. adaptively (a.) refined meshes.
  From top left to bottom right the plots correspond to $t=10^0,10^{-1},10^{-2},10^{-3}$.}
  \label{fig:Lshape}
\end{figure}
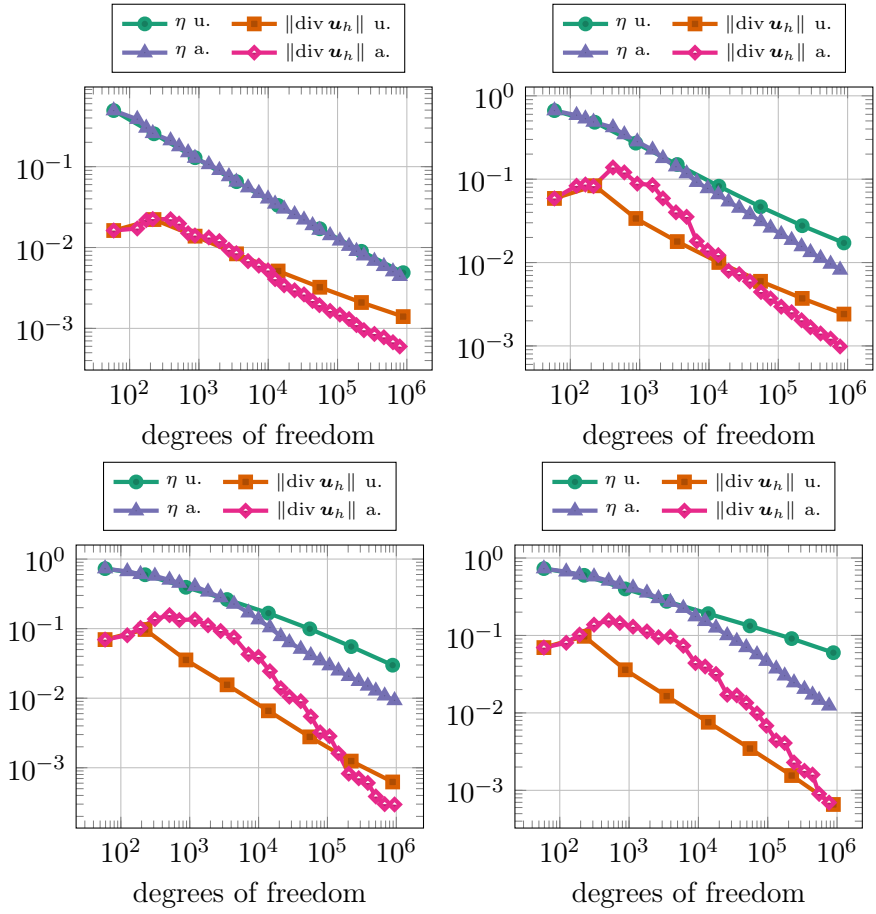

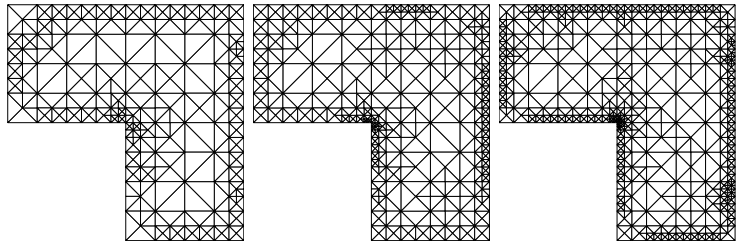
\begin{figure}
  \begin{tikzpicture}
\begin{axis}[hide axis,
width=0.49\textwidth,
    axis equal,
]

\addplot[patch,color=white,
faceted color = black, line width = 0.2pt,
patch table ={data/elements_t1e-02_nE403.dat}] file{data/coordinates_t1e-02_nE403.dat};
\end{axis}
\end{tikzpicture}
\begin{tikzpicture}
\begin{axis}[hide axis,
width=0.49\textwidth,
    axis equal,
]

\addplot[patch,color=white,
faceted color = black, line width = 0.2pt,
patch table ={data/elements_t1e-02_nE615.dat}] file{data/coordinates_t1e-02_nE615.dat};
\end{axis}
\end{tikzpicture}
\begin{tikzpicture}
\begin{axis}[hide axis,
width=0.49\textwidth,
    axis equal,
]

\addplot[patch,color=white,
faceted color = black, line width = 0.2pt,
patch table ={data/elements_t1e-02_nE964.dat}] file{data/coordinates_t1e-02_nE964.dat};
\end{axis}
\end{tikzpicture}
  \caption{Adaptively refined meshes with $t=10^{-2}$. From left to right the meshes contain $403$, $615$, and $964$ elements respectively.}
  \label{fig:Lshape:meshes}
\end{figure}

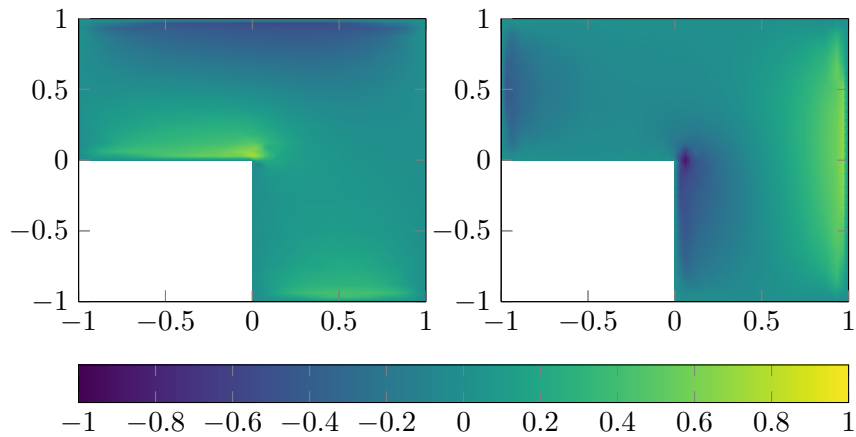
\begin{figure}
  \begin{tikzpicture}
  \begin{groupplot}[
      view={0}{90},
      width=0.49\textwidth,
    point meta min=-1,
    point meta max=1,
    colormap/viridis,
    group style = {group size = 2 by 1
    }]
    \nextgroupplot[colorbar horizontal,
    every colorbar/.append style={width=
        2*\pgfkeysvalueof{/pgfplots/parent axis width}+
        \pgfkeysvalueof{/pgfplots/group/horizontal sep}}]
        \addplot3[hide axis,patch,shader=interp] table{data/sol1_t1e-02.dat};
    \nextgroupplot
    \addplot3[hide axis,patch,shader=interp] table{data/sol2_t1e-02.dat};
  \end{groupplot}
\end{tikzpicture}
  \caption{Solution components of the approximated solution $\uu_h$ with $t=10^{-2}$ on a mesh generated by the adaptive algorithm with $2188$ elements.}
  \label{fig:Lshape:sol}
\end{figure}

\begin{figure}
  \includegraphics[width=0.7\textwidth]{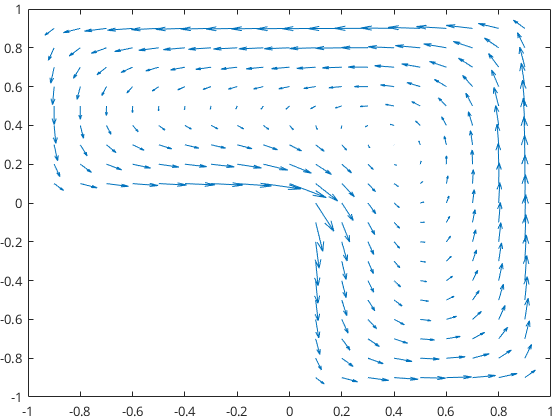}
  \caption{Velocity field in the interior of the L-shaped domain with $t=10^{-2}$.}
  \label{fig:Lshape:vecField}
\end{figure}

Figure~\ref{fig:Lshape:meshes} presents a sequence of meshes generated by the adaptive loop. 
One observes refinements towards the boundary edges as well as towards the reentrant corner. 

Finally, Figure~\ref{fig:Lshape:sol} visualizes the two solution components of the approximated solution $\uu_h$ for the problem described in Section~\ref{sec:num:Lshape} with $t=10^{-2}$ on a mesh generated by the adaptive algorithm with $2188$ elements.
Both solutions seem to have boundary layers on different edges and both components seem to have a singularity at the reentrant corner as expected. These effects are captured by the adaptive algorithm as can also be seen on the generated meshes, see Figure~\ref{fig:Lshape:meshes}. 
Figure~\ref{fig:Lshape:vecField} shows the corresponding velocity field.

\bibliographystyle{abbrv}
\bibliography{literature}

\end{document}